\documentclass[12pt]{amsart}

\usepackage{amsthm}
\usepackage{amsfonts}
\usepackage{graphicx}

\newcommand{\R}{{\mathbb{R}}}

\newtheorem{theorem}{Theorem}[section]

\newtheorem{lemma}[theorem]{Lemma}
\newtheorem{proposition}[theorem]{Proposition}
\newtheorem{conjecture}[theorem]{Conjecture}

\newtheorem{definition}[theorem]{Definition}

\begin{document}
\title{New examples of tunnel number subadditivity}
\author{Trent Schirmer}
\email{trenton-schirmer@uiowa.edu}

\begin{abstract}
If the tunnel number of a knot $K$ is denoted $t(K)$, a pair of knots $K_1,K_2$ is said to be subadditive if $t(K_1)+t(K_2)>t(K_1 \# K_2)$.  In \cite{schulschar} Scharlemann and Schultens defined the degeneration ratio to be $d(K_1,K_2)=1-\frac{t(K_1\# K_2)}{t(K_1)+t(K_2)}$, and proved that $d(K_1,K_2)\leq 3/5$.  However, the highest known degeneration ratio known for a pair of knots is just $2/5$.  We use free decompositions to construct links which experience degeneration approaching $3/7$ when the connect sum is taken with certain knots.  These links can be modified to yield a family of knots whose members we conjecture to have the same property.
\end{abstract}

\maketitle

\section{Introduction}

The focus of this paper is on the behavior of tunnel number under the operation of connect sum of knots and/or links, which has received substantial attention.  It is known that tunnel number can be additive $t(K_1\# K_2)=t(K_1)+t(K_2)$, subadditive $t(K_1 \# K_2)<t(K_1)+t(K_2)$, or superadditive $t(K_1\# K_2)>t(K_1)+t(K_2)$, but within limits.\\

For the case of superadditivity one has, at most, $t(K_1 \# K_2)=t(K_1)+t(K_2)+1$.  Morimoto, Sakuma, and Yokota \cite{mor2}, found pairs of tunnel number one knots whose connect sum has tunnel number $3$, and soon after Moriah and Rubinstein showed that superadditive pairs of knots exist of arbitrarily large tunnel number, as a corollary to the main result of their paper \cite{morrub}.  Later Rieck \cite{rieck} generalized the main result of \cite{morrub}, and later still Kobayashi and Rieck extended Moriah and Rubinstein's corollary on superadditive knot pairs \cite{kobri}, using it to provide a counterexample to an important conjecture of Morimoto's coming from \cite{mor4}.  Kobayashi and Rieck's result on superadditivity plays an important role in our construction as well; we state a weakened version of it in Theorem 4.4 below.\\

Morimoto found a subadditive pair of knots in \cite{mor3}, and Kobayashi soon after found that the quantity $t(K_1)+t(K_2)-t(K_1\#K_2)$ can be arbitrarily large \cite{kob}.  On the other hand, Scharlemann and Schultens showed that, for prime knots, the ``degeneration ratio'' $d(K_1,K_2)=1-\frac{t(K_1\# K_2)}{t(K_1)+t(K_2)}\leq \frac {3}{5}$ \cite{schulschar}.  A more complete account of this and many other problems related to tunnel number can be found in Moriah's survey \cite{ymor}.\\

More recent work of Morimoto in \cite{mor1} uses free decompositions to describe some properties that certain knot pairs could have that would allow them to achieve a degeneration ratio of $2/5$, and recently Joao Nogueira found knots satisfying just these properties in his dissertation \cite{nog}.  Below we generalize Morimoto's construction considerably and use it to describe a family of knots which we conjecture to asymptotically approach a degeneration ratio of $3/7$.  Included in this family are the knots of Nogueira (which were independently discovered by the author), and do achieve this property.  As a step in the construction we build links of arbitrarily large tunnel number which we prove do in fact achieve the expected degeneration ratios when the connect sum is taken with appropriate tuples of knots.  This in conjunction with Nogueira's result makes a positive answer to Conjecture 4.7 appear very likely.\\

\section{Notation, definitions, and some handlebody basics}

From now on $L$ shall be a link embedded in a compact orientable 3-manifold $M$ and, more specifically, $K$ shall be a knot. If $X$ is a polyhedron in the PL manifold $Y$, let $N(X)$ denote a closed regular neighborhood of $X$ in $Y$, and let $E(X)$ denote $\overline{Y\setminus N(X)}$.\\

\begin{definition}
An {\em unknotting system} for $L$ is a collection of arcs $T=\{t_1\cup \cdots t_n\}$ properly embedded in $E(L)$ such that $E(L\cup t_1 \cup \cdots t_n)$ is a handlebody.  The minimal cardinality of an unknotting system for $L$ shall be denoted $t(L)$, the {\em tunnel number} of $L$.\\
\end{definition}


\begin{definition}
If $V$ is a handlebody, an embedded graph $X$ in $V$ is said to be a {\em  spine} of $V$ if $E(X)$ is homeomorphic to $(\partial V)\times I$, where $I=[0,1]\subset \R$.  A subgraph $X'$ of a spine $X$ in $V$ will be called a {\em subspine} if it contains no contractible components, and in the special case that $X'$ is a collection of loops it will be called a {\em core} of $V$.\\
\end{definition}


\begin{definition}
A {\em compression body} is a handlebody or a thickened surface with 1-handles attached to one of its boundary components.  Equivalently it is a handlebody with an open regular neighborhood of a subspine removed.  Given a compression body $H=\overline{V\setminus N (X)}$, we define $\partial_+H=\partial V$ and $\partial_-H=\partial N(X)$.\\
\end{definition}

We allow this to include the case of the empty subspine and a full spine, thus handlebodies and thickened surfaces count as compression bodies.  \\


\begin{definition}

A {\em Heegaard splitting} of a compact orientable $3$-manifold $M$ is a decomposition $M=H_1\cup H_2$ where each $H_i$ is a compression body and $\partial_+ H_1=\partial_+ H_2=H_1\cap H_2$ is called a {\em Heegaard surface}.  The {\em Heegaard genus} $H(M)$ of a manifold is the minimal genus of a Heegaard surface for $M$.\\

\end{definition}


\begin{definition}
A collection of simple closed curves ${\bf C}=\{C_1, \cdots ,C_n\}$ embedded in $\partial V$ is called {\em primitive} if there exists a disjoint collection of compressing disks ${\bf D}=\{D_1, \cdots , D_n\}$ for $V$ such that $|C_i\cap D_j|=\delta_{ij}$ (where $\delta_{ij}$ is the Kronecker delta).  The disks ${\bf D}$ are said to be {\em dual} to ${\bf C}$. \\
\end{definition}


\begin{proposition}
A primitive collection ${\bf C}=\{C_1, \cdots , C_n\}$ of curves embedded in the boundary of a handlebody $V$ is isotopic to a core of $V$.\\
\end{proposition}

\begin{proof}
It suffices to show that ${\bf C}$ can be isotoped into $V$ so that $E({\bf C})$ becomes a compression body.  Let ${\bf D}=\{D_1, \cdots , D_n\}$ be a collection of disks dual to ${\bf C}$. Then ${\bf E}=\overline{\partial N(C_1\cup \cdots \cup C_n \cup D_1 \cup \cdots \cup D_n) \setminus \partial V}$ is a collection of compressing disks which cut $V$ into a handlebody $\tilde{V}$ and a collection of solid tori $T_i$, each having a core isotopic to exactly one element $C_i\in{\bf C}$.  Thus, we see that $E({\bf C})$ is obtained by attaching a collection of thickened tori $\overline{T_i\setminus N(C_i)}$ to a handlebody (or ball) $\tilde{V}$ along disks, and so must be a compression body.\\
\end{proof}


\begin{proposition}
If ${\bf C}$ is a primitive collection of simple closed curves on the boundary of a handlebody $V$, and $h:N({\bf C})\cap \partial V\rightarrow \partial W$ is an orientation reversing embedding into the boundary of a handlebody $W$, then $V\cup_hW$ is a handlebody.\\
\end{proposition}

\begin{proof}
Let $\bf{D}$ be the collection of disks dual to ${\bf C}$.  Then if $H=N({\bf C} \cup {\bf D})$, $W'=H\cup_h W$ is homeomorphic to $W$ and thus is a handlebody. But $V\cup_hW$ is obtained from $W'$ by attaching the handlebodies $E({\bf C} \cup {\bf D})$ to $W'$ along disks, and so is also a handlebody.\\
\end{proof}

\section{Free decompositions}

Free tangle decompositions were introduced into the literature by Kobayashi \cite{kob} in order to prove the existence of knot pairs whose tunnel number experiences arbitrarily large degeneration under connect sum.  In this section we define a slighly more general notion of free decompositions and prove some useful propositions about them.  

\begin{definition}
Let $M$ be a compact orientable $3$ manifold with non-empty boundary, and let $T=\{t_1,\cdots , t_n\}$ be a collection of arcs properly embedded in $M$.  The pair $(M,T)$ is called a tangle in $M$, and it is {\em free} if $E(T)$ is a handlebody.\\
\end{definition}

Observe that an unknotting system for $L$ is a free tangle in $E(L)$.  It is useful to specialize Definition 3.2 as follows:\\


\begin{definition}
A tangle $(M,T)$ is said to be {\em trivial} if it satisfies the following conditions:

\begin{itemize}
\item $(M,T)$ is free
\item For every $t\in T$, there is a disk $D$ embedded in $M$ such that $\partial D=t\cup\alpha$, where $\alpha=D \cap \partial M$ is an arc with $\alpha \cap t =\partial t =\partial \alpha$, and $int(D)\cap (t_1\cup \cdots \cup t_n)=\emptyset$.

\end{itemize} 

\end{definition}

The collection ${\bf D}$ of disks associated with a trivial tangle $T$ above can always be chosen disjoint.  In the next definition, the condition that $S$ be strongly separating means that each component of $E(S)$ can be labeled with a $+$ or $-$ in such a way that no pair of adjacent components share a common sign.\\


\begin{definition}
Let $S$ be a strongly separating closed surface in a closed orientable $3$-manifold $M$, and let $L$ be a link in $M$ transverse to $S$.  If, for each component $V_i$ of $E(S)$, the tangle $(V_i,L\cap V)=(V_i,T_i)$ is free, $S$ is said to be a {\em free decomposing surface} for $L$, and $(S,(V_1,T_1),\cdots , (V_n,T_n))$ is a {\em free decomposition} of $L$.
\end{definition}


\begin{definition}
A free decomposing surface $F$ for $L\subset M$ is called an $|L\cap F|/2$-{\em bridge surface} if $(V,L\cap V)$ is a trivial tangle for each component $V$ of $E(F)$. This forces $F$ to be connected Heegaard surface of $M$.  The minimal $n$ for which a genus $g$, $n$-bridge surface of $L$ exists is called the {\em genus $g$ bridge number} of $L$, denoted $b(g,L)$.  In the case $M=S^3$ we write $b(0,L)=b(L)$, which is called simply the {\em bridge number} of $L$.\\
\end{definition}

The following somewhat technical lemma is of central importance to the work that follows; it can be viewed as an extension of Proposition 3.5 of Kobayashi's paper \cite{kob} to our more general setting:\\


\begin{lemma}
Suppose $L$ is a link in a closed orientable $3$-manifold $M$ with components $L_1,\cdots ,L_n$ and a free decomposition $(S,(V_1,T_1), \cdots , (V_k,T_k))$, and that for each $L_i$ there is an arc $y_i$ in $L_i\cap V_j$ for some $j$ which cobounds a disk with an arc on $\partial V_j$.  Let $\{K_i\}_{i=1}^n$ be a collection of knots in $M_i$ with bridge surfaces $\{F_i\}_{i=1}^n$ satisfying $|K_i\cap F_i|=|L_i\cap S|$ for each $1\leq i\leq n$.  Then the connect sum $L'\subset M'=M\# M_1 \# \cdots \# M_n$ of $L$ with the collection $\{K_i\}_{i=1}^n$ (so that each $K_i$ connects along the component $L_i$) can be taken in such a way that the resulting glued surface $\overline{S\setminus N(L)}\cup \overline{F_1\setminus N(K_1)} \cup \cdots \cup \overline{F_n\setminus N(K_n)}=S'$ becomes a closed connected Heegaard surface of $E(L')$.\\

\end{lemma}

\begin{proof}
For each $K_i$, let $E(F_i)=W_i^1\cup W_i^2$, and choose any component $w_i$ of $K_i\cap W_i^1$.  Let $R_i^1$ be the sphere $\partial N(K_i\setminus w_i)$, and $R_i^2$ be the sphere $\partial N(L_i \setminus y_i)$.  Then the connect sum in question can be obtained by attaching $E(K_i)$ to $E(L_i)$ via an orientation reversing map $h_i: (R_i^1,F_i\cap R_i^1)\rightarrow (R_i^2,  S\cap R_i^2)$, as pictured in Figure 1.\\

\begin{figure}
\begin{center}
\includegraphics[width=0.8\textwidth]{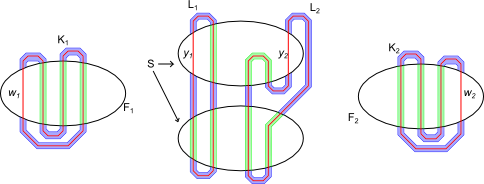}
\end{center}
\caption{Schematic diagram for the case $n=2$.  Knots and links are in red, free decomposing surfaces are in black, and the spheres $R_i^j$ are represented by light blue and green tubes running along the knots and links, according to the pattern of their identification via the maps $h_i$.}
\end{figure}

We assumed in Definition 3.3 that $S$ strongly separates $M$ into non-adjacent $+$ and $-$ components, without loss of generality suppose ${\bf V_+}=\{V_1,\cdots V_l  \}$ and ${\bf V_-}=\{V_{l+1},\cdots , V_k\}$ are the sets of $+$ and $-$ components, respectively.  If, for a given $i$, $y_i$ lies in a $+$ component of $E(S)$, it is not difficult to see that $h_i(R_i^1\cap W_i^1)=R_i^2\cap (V_1\cup \cdots \cup V_l)$, and  $h_i(R_i^1\cap W_i^2)=R_i^2\cap (V_{l+1}\cup \cdots \cup V_k)$.  Similarly, if $y_i$ lies in a $-$ component, then $h$ glues $W_i^1$ only to $-$ components, and $W_i^2$ only to $+$ components.  Thus, once the gluings are performed for each $i$, $1\leq i \leq n$, the complementary components of the resulting surface $S'$ can be labeled with a $+$ or $-$ consistently with the original labeling of the components $V_i$, so $S'$ is strongly separating in $M'$ as required.\\

Note that $L\cup S$ is necessarily connected since $\overline{V_j\setminus N(L)}$ is a handlebody for each $j$ and thus must have a single boundary component.  Moreover, if a component $S_i$ is connected to a component $S_j$ of $S$ by $L_k$, then $h_i$ glues $F_k$, which is connected, to both $S_j$ and $S_i$, making them both part of the same component in $S'$.  But then it follows that $S'$ is connected and, by the previous paragraph, separating.  Thus $E(S')$ has two components $Y_1, Y_2$; we now will show that each of them is a handlebody.\\

Recall that by definition each of the pieces $\overline{W_i^1\setminus N (K_i)}$, $\overline{W_i^2\setminus N (K_i)}$, and $\overline{V_j\setminus N(L)}$ is a handlebody.  Moreover since the arcs $y_i$ and $w_i$ cobound disks on the boundaries of those pieces in which they are embedded, the pieces $\tilde{W}_i^1=\overline {W_i^1\setminus N(K_i\setminus w_i)}$ and $\tilde{V}_j=\overline{V_j\setminus N (L\setminus \{y_1\cup\cdots \cup y_n\})}$ are all handlebodies as well.\\

Each of $Y_1$ and $Y_2$ is obtained by gluing some subcollection of the handlebodies $\tilde{W}_i^1$, $\tilde{W}_i^2=\overline{W_i^2\setminus N (K_i)}$ to some subcollection of the handlebodies $\tilde{V}_j$ along some subcollection of the annuli and disks appearing in $R_i^1\setminus F_i$.  The core curves of the annuli in  $R_i^1\setminus F_i$ form primitive collections of simple closed curves on $\partial \tilde{W_i^1}$ and $\partial \tilde{W_i^2}$, so by Proposition 2.7 the result of gluing $\tilde{W_i^1}$ and $\tilde{W_i^2}$ to the $\tilde{V}_j$ along these annuli is a collection of handlebodies.  Since each of $Y_1$ and $Y_2$ is then obtained from this by identifying disks on $\partial \tilde{W_i^1}$ and $\partial \tilde{W_i^2}$ to disks on $\partial \tilde{V}_j$, i.e. by attaching 1-handles, it follows that they too are handlebodies.\\

Thus $S'$ is a Heegaard surface for $M'$, so by Proposition 2.6 it is enough to show that $L'$ can be isotoped onto a primitive collection of simple closed curves on $S'$.  But each component $L_i'$ of $L'$ is just the union $y_i \cup w_i$.  By our hypotheses on these subarcs each $L_i'$ can thus be isotoped onto $S'$, and indeed the closure of either disk in $R_i^1\setminus F_i$ will serve as a dual compressing disk for $L_i'$.\\
  
\end{proof}

Free decompositions of links place upper bounds on their tunnel number, as was shown by Morimoto in the case when the decomposing surface is a single sphere in \cite{mor1}.  His methods do not extend to our more general case.  However a bound does exist in terms of arbitrary free decomposing surfaces.\\


\begin{proposition}
Let $L\subset M$ be a link with a free decomposing surface $S$.  Then $H(E(L))\leq 1+|L\cap S|-\frac{\chi(S)}{2}$.\\
\end{proposition}

\begin{proof}
For each component $L_i$, pick a single small subarc $\gamma_i$ of some component of $L_i\setminus S$, and add the spheres $S_i=\partial N (\gamma_i))$ to $S$ to obtain a new free decomposing surface $\tilde{S}$ satisfying the hypothesis of Lemma 3.5.  Corresponding to each $L_i$, let $K_i$ be the unknot in $S^3$ together with an $|L_i\cap \tilde{S}|/2$ bridge sphere $F_i$.\\

Taking the connect sum of $L$ with the $K_i$ as in Lemma 3.5 yields back $L$ again, together with a Heegaard surface for $S'$ for $E(L)$.  We now compute

\begin{center}
$\chi (S')= \chi (\tilde{S}\setminus L)+\displaystyle\sum\limits_i \chi (F_i\setminus K_i)$
\end{center}

and since

\begin{center}
$\chi(\tilde{S}\setminus L)=\chi(S\setminus L)$, $\chi(F_i\setminus K_i)=2-|K_i\cap F_i|=-|L_i\cap S|$

\end{center}

we obtain

\begin{center}
$g(S')=1-\frac{\chi (S')}{2}=1-\frac{\chi(S\setminus L)-|L\cap S|}{2}$
\end{center}

and, since $\chi(S\setminus L)=\chi(S)-|L\cap S|$, we deduce the desired inequality.\\

\end{proof}

The small spheres added to $S$ to obtain $\tilde{S}$ can always be chosen so that every component of $L$ lies on the same side of $S'$.  So we can in fact conclude $t(L)\leq |L\cap S|-\frac{\chi(S)}{2}$.

\section{New examples of tunnel number subadditivity}

We now come to the object of our exercises.\\

\begin{definition}
Let $L$ be a link of $n$ components $L_1, \cdots, L_n$, let $K_1,\cdots , K_n$ be a collection of knots, and let $L\#(K_1, \cdots , K_n)$ be the connect sum taken so that $K_i$ connects along $L_i$.  The {\em degeneration ratio} is then given as follows:

\begin{center}
$d_L(K_1,\cdots, K_n)=\frac{t(L)+t(K_1)+\cdots + t(K_n)-t(L\#(K_1,\cdots, K_n))}{t(L)+t(K_1)+\cdots + t(K_n)}$
\end{center}

\end{definition}

The links we find below that admit high degeneration are of the following kind:\\


\begin{definition}
A free decomposing surface $S$ for a link $L$ is {\em optimal} if $t(L)=|L\cap S|-\frac{\chi(S)}{2}$.\\
\end{definition}


\begin{proposition}
If an $n$-component link $L$ admits an optimal free decomposing surface $S$, then there exists a collection of knots $K_1, \cdots ,K_n$ such that 

\begin{center}
$d_L(K_1, \cdots, K_n)\geq \frac {|S\cap L|}{3|S\cap L|-\chi(S)}$\\

\end{center}

\end{proposition}

\begin{proof}
Follow the proof of Proposition 3.6 exactly, except let your $K_i$ be $|L_i\cap \tilde{S}|/2$ bridge, tunnel number $|L_i\cap \tilde{S}|/2-1$ knots in $S^3$, where $\tilde{S}$ is the same modified surface described there.  That such knots do exist can be seen as follows.  First, it is a well-known result of Schubert that $b(K_1\# \cdots \# K_m)= b(K_1)+\cdots + b(K_m)-m+1$ for any collection $m$ knots.  Second, by a result of Scharlemann and Schultens \cite{schulschar2}, $t(K_1\# \cdots \# K_m)\geq m$ for any collection of $m$ knots.  Finally, it is easy to see that $t(K)<b(K)$ for any knot $K$.  It follows that any $n-1$-fold connect sum of $2$-bridge knots is an $n$-bridge, tunnel number $n-1$ knot.\\\\

By hypothesis $t(L)=|L\cap S|-\frac{\chi(S)}{2}$ and $t(K_i)=|L_i\cap S|/2$, while the remark at the end of Proposition 3.6 shows that $t(L\#(K_1, \cdots, K_n)\leq |L\cap S|-\frac{\chi(S)}{2}$.  Plugging in to the formula of Definition 4.1 yields the result.\\

\end{proof}

To construct our family of optimal links, we will need the following (weakened version) of an important theorem of Kobayashi and Rieck.\\


\begin{theorem}\cite{kobri}
For any collection of positive integers $\{m_1,\cdots , m_n\}$ there exists a collection of knots $\{ K_1, \cdots , K_n\}$ in $S^3$ satisfying $t(K_i)=m_i$ and $t(K_1 \# \cdots \# K_n)= n-1+\displaystyle\sum\limits_i m_i$.\\
\end{theorem}


\begin{theorem}
For all integers $n>0$, there exist $n+1$ component links $L$ and knots $K_1,\cdots, K_{n+1}$ in $S^3$ such that

\begin{center}
$d_L(K_1,\cdots , K_{n+1})\geq \frac{3n-1}{7n-2}$\\
\end{center}

\end{theorem}

\begin{proof}
Let $J_1,\cdots , J_{2n}$ be a collection of knots in $S^3$ satisfying $t(J_i)=1$ for all $i$ and $t(J_1\# \cdots \# J_{2n})=4n-1$, which exist by Theorem 4.4.  Let $S=S_1\cup \cdots \cup S_{2n-1}$ be a collection of decomposing spheres for the connect sum $J=J_1\# \cdots \# J_{2n}$ satisfying the property that, for each $i$, $S_i$ bounds a ball $B$ satisfying $B\cap S=S_1\cup\cdots \cup S_{i-1}$, i.e. let the $S_i$ be nested as in Figure 2. $W_1, \cdots , W_{2n}$ be the components of $E(S\cup J)$, labeled so that $E(J_i)\cong W_i$ and $W_i\cap W_{i+1}=S_i$. \\

\begin{figure}
\begin{center}
\includegraphics[width=0.8\textwidth]{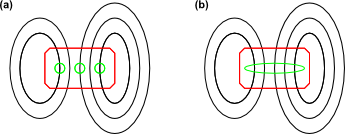}
\end{center}
\caption{Schematic diagram depicting the links constructed in Theorems 4.5 and 4.6.  The decomposing spheres are indicated in black, the component of $L$ coming from the connect sum of superadditive knots is in red, and the link components arising from sliding tunnels are in green.}
\end{figure}

Since $t(J_i)=1$, each $W_i$ admits an arc $t_i$ such that $(W_i,t_i)$ is free.  Moreover, if one marks a pair of points $a_i,b_i$ on the surface of $S_i\setminus N(J)$ for each odd $i$, then the arcs $t_i$ and $t_{i+1}$ can be properly isotoped in $W_i$ and $W_{i+1}$, respectively, so that their endpoints lie on these marked points.  The result will be that the union $t_i\cup t_{i+1}=L_{(i+1)/2}$, forms a closed loop in $S^3$ for each $i\equiv 1 \pmod {2}$, see Figure 2(a).  Let $L$ be the link $J\cup L_1 \cup \cdots \cup L_n$.  Clearly $t(L)\geq t(J)$, since any Heegaard splitting for $L$ is also a Heegaard splitting for $K$, and by Proposition 3.6, $t(L)\leq t(J)$ as well, since $S$ is a free decomposing surface for $L$, in fact an optimal one.  The inequality now follows from Proposition 4.3.\\

\end{proof}

\begin{theorem}
For all integers $n>0$ there are two component links $L\subset S^3$ with $t(L)=3n$ and pairs of knots $K_1,K_2\subset S^3$ such that $d_L(K_1,K_2)\geq 2/5$.\\
\end{theorem}

\begin{proof}
The construction is nearly identical to that of Theorem 4.5, except we start by taking the connect sum $J=J_1\# \cdots \# J_{n+1}$ with $t(J_1)=t(J_{n+1})=1$, $t(J_i)=2$ for $1<i<n+1$, and $t(J)=3n$.  The difference is that after decomposing the connect sum along nested spheres as in Theorem 4.5, the tunnels in $W_i$ can be slid together to form a single loop instead of many, as in Figure 2(b).\\

\end{proof}

Let $L$ be a link with free decomposing surface $S$.  Then we may regard $N(S)\cap L$ as a collection of trivial braids in $N(S)\cong S\times I$ (one for each component of $S\times I$).  Substituting an arbitrary collection of braids for $N(S)\cap L$ yields another link for which $S$ is also a free decomposing surface, see Figure 3 for an example.\\

\begin{figure}
\begin{center}
\includegraphics[width=0.8\textwidth]{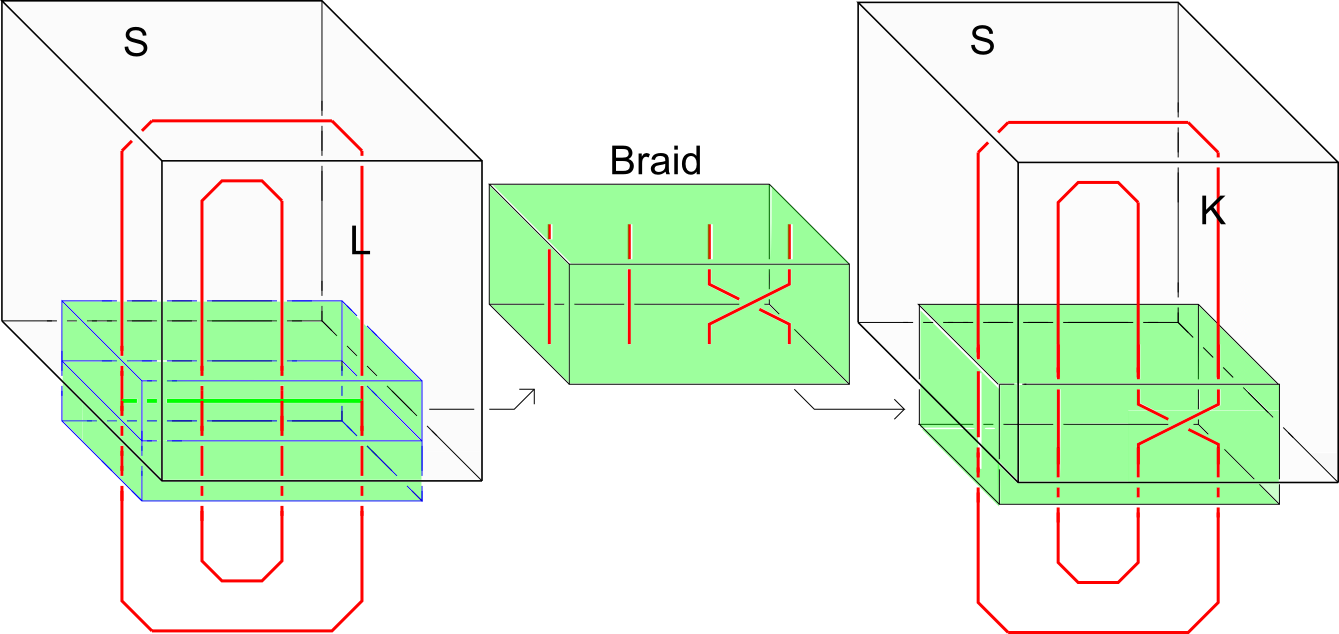}
\end{center}
\caption{Braid Substitution}
\end{figure}

\begin{conjecture}

There exist knots $K$ obtained from the links of Propositions 4.5 and 4.6 by braid substitutions which are optimal.\\

\end{conjecture}

We conclude by pointing out that Nogueira \cite{nog} has already given an affirmative answer to Conjecture 4.7 in the case $n=1$ of Theorems 4.5 and 4.6 (which coincide).

\section{Acknowledgments}

I would like to thank Maggy Tomova and Charlie Frohman for all of their time, help, and advice.

\end{document}